\documentclass{article}

\usepackage[english]{babel}
\usepackage{color}
\usepackage{amsfonts}
\usepackage{enumerate}
\usepackage{amsthm,amsmath,amssymb}
\usepackage{cite}

\newtheorem{theorem}{Theorem}[section]
\newtheorem*{theorem*}{Theorem}
\newtheorem{proposition}[theorem]{Proposition}
\newtheorem{corollary}[theorem]{Corollary}
\newtheorem{lemma}[theorem]{Lemma}
\newtheorem{conjecture}{Conjecture}
\theoremstyle{definition}

\newtheorem{definition}[theorem]{Definition}

\providecommand{\abs}[1]{\lvert#1\rvert}

\newcommand\restr[2]{\ensuremath{\left.#1\right|_{#2}}}
\renewcommand{\subset}{\subseteq}

\newcommand{\tl}{t_{\lambda}}
\newcommand{\tm}{t_{\mu}}

\newcommand{\Q}{\mathbb{Q}}

\newcommand{\Qp}{\mathbb{Q}_p}

\newcommand{\Z}{\mathbb{Z}}

\newcommand{\N}{\mathbb{N}}

\newcommand{\R}{\mathbb{R}}

\newcommand{\Lm}{\mathcal{L}}

\newcommand{\cO}{\mathcal{O}}

\newcommand{\ord}{\mathrm{ord}}

\newcommand{\Lring}{\Lm_{\text{ring}}}

\title{Differentiation in $P$-minimal structures\\ and a $p$-adic Local Monotonicity Theorem}
\author{Tristan Kuijpers and Eva Leenknegt}
\date{}

\begin{document}
\maketitle
\begin{abstract}	
We prove a $p$-adic, local version of the Monotonicity Theorem for $P$-minimal structures. The existence of such a theorem was originally conjectured by Haskell and Macpherson. We approach the problem by considering  the first order strict derivative. In particular, we show that, for a wide class of $P$-minimal structures, the definable functions $f: K \to K$ are almost everywhere strictly differentiable and satisfy the Local Jacobian Property. 	
\end{abstract}

\section{Introduction}
A major tool in the study of $o$-minimal structures is the Monotonicity Theorem, which states that for any $o$-minimal function $f:D \subseteq R\to R$, there exists a finite partition of $D$, such that on each part, $f$ is either constant or continuous and strictly monotone (see e.g. van den Dries \cite{vdd-98}).

When Haskell and Macpherson developed their theory of $P$-minimality \cite{has-mac-97} (a $p$-adic counterpart to the concept of $o$-minimality), a question that came up naturally was whether there would exist a $p$-adic version of this theorem. Of course, this question only makes sense if one can find a reasonable translation of the concept of \emph{monotonicity} to the $p$-adic context. 

Say that $z$ lies {between} $x$ and $y$ if $z$ is contained in the smallest ball that contains both $x$ and $y$. Using this notion, one can formulate a concept of monotonicity that works both in the real and the $p$-adic setting:
\begin{definition}
Let $F$ be a topological field. A function $f: F \to F$ is \emph{monotone} if, whenever $z$ lies between $x$ and $y$, then also $f(z)$ lies between $f(x)$ and $f(y)$.
\end{definition}
For ultrametric fields, this condition is equivalent to 
\begin{equation*} \label{eq:Ccond}
|x-z|\leqslant |x-y| \Rightarrow |f(x) - f(z)| \leqslant |f(x) - f(y)|.
\end{equation*}
Observe that if a function $f$ is monotone and $f(x) = f(y)$, then $f$ is constant \emph{between} $x$ and $y$.
A more detailed exploration can be found in \cite{sch-84}.

Extending this idea further, the following would be a natural translation to the $p$-adic context of (local) strict monotonicity. (Remember that in the real case, a strictly monotone function $f$ is either strictly increasing or decreasing, and hence we get a bijection between the domain of $f$ and the image of $f$.)
\begin{definition}  
Let $F$ be an ultrametric field.
A function $f: X \subseteq F \to F$ is said to be \emph{locally strictly monotone} on $X$ if for all $a \in X$, there exist balls $B_1$, $B_2$ such that $a \in B_1 \subseteq V$, $f$ maps $B_1$ bijectively onto $B_2$, and for all $x,y,z \in B_1$, 
	\[|x-z| < |x-y| \Rightarrow |f(x) -f(z)|<|f(x)-f(y)|.\]
\end{definition}
Having this notion in mind, Haskell and Macpherson \cite{has-mac-97} stated the following conjecture, which can be considered as a {local version} of the Monotonicity Theorem for $p$-adically closed fields $K$.
\begin{conjecture}
Let $f: X \subseteq K \to K$ be a function definable in a $P$-minimal structure $(K, \Lm)$. There exist definable disjoint subsets $U, V$ of $X$, with $X \setminus (U \cup V)$ finite, such that 
\begin{itemize} 
	\item[(a)] $\restr{f}{U}$ is locally constant, 
	\item[(b)] $\restr{f}{V}$ is locally strictly monotone on $V$.
\end{itemize}
\end{conjecture}
Unfortunately, Haskell and Macpherson could only prove a weaker version of this conjecture. The main motivation of this paper is to give a full proof (we even obtain a slightly more precise result). The key to the problem is the existence of the first order (strict) derivative of $P$-minimal functions.

\subsection{Differentiation in the $p$-adic context}
Ever since the end of the 19th century, the theory of differentiation (and integration) of real functions has been well established. However, the picture is not quite as rosy when considering $p$-adic functions. Whereas real analysis has become a basic tool (even for non-mathematicians), $p$-adic analysis is more subtle for several reasons.

One of the consequences of the ultrametric topology is that the mean value theorem no longer holds on $p$-adic fields. Because of this, even if we restrict to the category of \emph{nice} functions that have a continuous derivative, examples can be found of functions that behave badly: an injective function that has  a derivative which is zero everywhere, or a function that has nonzero derivative, yet is not injective in any neighbourhood of zero (see e.g. examples 26.4 and 26.6 in Schikhof's book \cite{sch-84}.)

To remedy some of the problems listed above, we will need to consider a stronger concept of differentiation.  A natural candidate is the following notion of strict differentiation (a detailed exposition of which can be found in Schikhof \cite{sch-84} or Robert \cite{rob-2000}):
\begin{definition}Let $X \subseteq K$ be an open set. A function $f:X \to K$ is \emph{strictly differentiable} at a point $a\in X$, with strict derivative $Df(a)$ if the limit
\[Df(a) = \lim_{(x,y) \to (a,a)} \frac{f(x)-f(y)}{x-y}\]
exists. 
\end{definition} 
To distinguish between both concepts, we use the notation $Df$ to refer to the strict derivative, and we write $f'$ for the normal derivative (as defined by Weierstrass). Obviously, $Df(a) = f'(a)$ whenever $Df(a)$ exists. 

If $f$ is strictly differentiable on an open set $U$, then $Df$ is continuous on $U$, by Proposition 27.2 of \cite{sch-84}, which means that a strictly differentiable function is automatically $C^1$.

Note that in the real case, every function $f$ for which $f'$ is continuous is automatically strictly differentiable, as can be seen easily by applying the mean value theorem. One way to look at it is that strict differentiation is a form of continuous differentiation where (consequences of) the mean value theorem are already built into the definition.

More recently, Bertram, Gl\"ockner and Neeb \cite{ber-glo-nee-04} developed a more general framework for differential calculus. When restricted to functions of one variable over ultrametric fields, their notion is equivalent to strict differentiation (see Section 6 of \cite{ber-glo-nee-04} for a comparison). There are also some differentiability results in the $C$-minimal setting, in Section 5 of \cite{kaz}. 

When using the stronger concept of strict differentiation, one can recover a number of the results that are foundational in real analysis. For example, a function that has nonzero strict derivative around a point $x_0$, will be injective on some neighborhood of $x_0$. However, some fundamental problems remain. For instance, the strictly differentiable function
\begin{equation} \label{badfunction}
g: \Q_p \to \Q_p: \sum_n a_np^n \mapsto  \sum_n a_np^{2n}
\end{equation}
is injective, and yet $Dg(x) = 0$ for all $x \in \Q_p$.   This example shows that, even with a stronger concept of differentiation, a nice theory will only be achievable if one also restricts to a more tame class of functions. Note that something similar has been done for real functions as well. Indeed, it is known that $o$-minimal functions $f: \R \to \R$ are (continuously) differentiable on a cofinite subset of $\R$ (see van den Dries \cite{vdd-98}).

Moreover, note that the function $g$ from \eqref{badfunction} is not $P$-minimal. Indeed, in $P$-minimal structures every infinite definable subset of the universe contains an open set. Clearly this is not true for $g(\Q_p)$.

\subsection{Basic definitions and facts}
Let us first review some basic facts about $P$-minimality (more details can be found in \cite{has-mac-97}). 

Let $\Lm$ be a language extending the ring language $\Lring= (+,-,\cdot, 0,1)$, and let $K$ be a $p$-adically closed field (that is, a field elementary equivalent, in the language of rings, to a finite field extension of $\Q_p$). The structure $(K, \Lm)$ is said to be \emph{$P$-minimal} if, for every elementary equivalent structure $(K',\Lm)$, any definable set $X\subset K'$ is $\Lring$-definable (with parameters from $K'$).

Examples of $P$-minimal structures include $p$-adic semi-algebraic sets and the structure of $p$-adic sub-analytic sets, as developed by Denef and van den Dries \cite{denef-vdd-88}. A function is said to be $P$-minimal if its graph is a $P$-minimal set.

By definable we always mean definable with parameters (and the underlying structure will be assumed to be $P$-minimal). A finite field extension of $\Qp$ will also be called a $p$-adic field. 

As a consequence, any definable subset of $K$ can be partitioned into a finite number of points and a finite number of open sets. This implies that every infinite definable subset of $K$ contains an open set, a fact which we will use quite often. The following lemma will also be used  extensively. 
\begin{lemma}[ Lemma 5.1 of \cite{has-mac-97}] \label{lemma_5.1}
	Let $f:X\subset K\to K$ be a function definable in a $P$-minimal structure $(K,\Lm)$. There is a cofinite subset $U\subset X$ such that $\restr{f}{U}$ is continuous.
\end{lemma}

We also mention the following theorem, which is a corollary of  Theorem 71.2 of \cite{sch-84}. This will be used to deduce strict differentiability from normal differentiability.
\begin{theorem}\label{thm_71.2}
	Let $K$ be a complete non-archimedean field and $X\subset K$ an open set. If $f:X\to K$ is differentiable, then the set 
	\[\{x\in X\mid f\text{ is strictly differentiable at }x\}\]
	is dense in $X$.
\end{theorem}

We write $\mathcal{O}_K$ for the valuation ring and $\Gamma_K$ for the value group. Let $\pi$ denote a fixed element with minimal positive valuation. Write $P_n$ for the set of nonzero $n$-th powers in $K$,  and $\lambda P_n$ for the coset $\{\lambda x\mid x\in P_n\}$, where $\lambda \in K$. Since $P_n$ has finite index in $K^\times$, one can choose a finite subset $\Lambda_n\subset K$ such that $K^{\times} = \cup_{\lambda\in\Lambda_n}\lambda P_n$. 

We let $B(x_0,\delta)$ denote the open ball with center $x_0$ and radius $\delta$, i.e. \[B(x_0,\delta) = \{x\in K\mid \abs{x-x_0}<\delta\}.\] We write $|K|=\{\abs{x}\mid x\in K\}$.  The notation $\abs{f'(x_0)}=+\infty$ means that 
\[\lim_{t\to 0}\left\lvert\frac{f(x_0+t)-f(x_0)}{t}\right\rvert = +\infty.\]

\subsection{Main results}
We cannot formulate our results for $P$-minimal structures in general. The main reason for this restriction is the following lemma, which will be essential. 
\begin{lemma}\label{lemma_locally_constant}
	Let $K$ be a $p$-adic field and let $f: X \subseteq K\to K$ be a differentiable function that is definable in a $P$-minimal structure. If $f'(x) =0$ for all $x \in X$, then there exists a finite partition of $X$ in parts $X_i$ such that $\restr{f}{X_i}$ is constant.
\end{lemma} 
A similar result can be found in \cite{sch-84}. A further generalisation to the context of $p$-adic integration is given in Proposition \ref{prop:integral}.
For real functions with connected domain, this is a simple  consequence of the mean value theorem. However, this does not hold for general $p$-adic functions: the function $g$ defined in \eqref{badfunction} provides a counterexample.

In our proof of Lemma \ref{lemma_locally_constant} we use that every open cover of $K$ has a countable subcover, hence the extra condition that $K$ is a $p$-adic field. We do not know whether this condition is essential.

We will show that our main results hold for any $P$-minimal structure satisfying the following additional condition:
\begin{definition}
A $P$-minimal structure $(K, \Lm)$ is said to be \emph{strictly $P$-minimal} if there exists a finite field extension $K'$ of $\Q_p$, such that $(K', \Lm)$ is $P$-minimal, and $K$ and $K'$ are elementarily equivalent as $\Lm$-structures.
\end{definition} 
Note that if Lemma \ref{lemma_locally_constant} would be true for all $p$-adically closed fields, then the condition of \emph{strict $P$-minimality} could be replaced by \emph{$P$-minimality} for all subsequent results. 

When working with general $p$-adically closed fields, it may happen that the value group $|K^\times|$ (considered as a multiplicative group) is not contained in $\mathbb{R}^\times$. In this case the limit of a function (and the derivative) can still be defined by the usual $(\epsilon,\delta)$-definition, the only difference being that $\epsilon$ and $\delta$ will be elements of $|K|$ rather than $\mathbb{R}$.

Another new result (see Proposition \ref{lemma_dim_graph}) that is crucial to our proofs is the fact that $\dim(\overline{X}\setminus X) <\dim(X)$ for any set $X$ definable in a $P$-minimal structure. This was already known for $o$-minimal structures, but is new in the $P$-minimal case. Using an improved version of a result by Haskell and Macpherson (where we eliminated the assumption of definable Skolem functions, see Lemma \ref{lemma5.5}), we were able to give a very short proof of this result.
\\\\
The first main result is a $p$-adic analogue of the result we mentioned earlier for $o$-minimal functions.
\begin{theorem}\label{thm_bijna_overal_afleidbaar}
Let $f:X \subseteq K\to K$ be a function definable in a  strictly $P$-minimal structure. Then $f$ is strictly differentiable  on a cofinite subset of $X$.
\end{theorem}
It will then be straightforward to show the second main result of this paper:
\begin{theorem}[Local Jacobian Property]\label{thm_LJP}
	Let $f:X\subset K \to K$  be a function definable in a strictly $P$-minimal structure. There exists a finite set $I\subseteq X$, and a finite partition of $X\setminus I$ into definable open sets $X_i$, such that either $\restr{f}{X_i}$ is constant on $X_i$, or the following holds on $X_i$: for every $x$ in $X_i$, there is an open ball $B\subset X_i$ containing $x$, such that the map $\restr{f}{B}$ satisfies the following properties:
\begin{enumerate}
\item[(a)] $\restr{f}{B}$ is a bijection, and $f(B)$ is a ball,
\item[(b)] $f$ is strictly differentiable on $B$ with strict derivative $Df$,
\item[(c)]$\lvert Df\rvert$ is constant on $B$,
\item[(d)] for all $x,y \in B$, one has that $\lvert Df\rvert\lvert x-y\rvert = \lvert f(x)-f(y)\rvert$.
\end{enumerate}
\end{theorem}

A global version of the above result was originally proven for semi-algebraic and sub-analytic sets by Cluckers and Lipshitz \cite{CLip}. Among other applications, it can be used in the study of $p$-adic and motivic integrals, see e.g. \cite{clu-loe-08}. The (global) Jacobian Property is also a valuable tool in the study of the geometry of definable sets (see e.g. \cite{clu-hal-2011} or \cite{clu-co-loe-2010}, where Lipschitz continuity was investigated). It is still an open question whether a global version of the Jacobian Property holds for general $P$-minimal structures.

Let us now return to the start of the introduction. The conjecture stated there is an immediate consequence of the Local Jacobian Property.
 If we combine this with Lemma \ref{lemma_5.1}, we obtain:
\begin{theorem}[$p$-adic Local Monotonicity]\label{thm:mono}
Let $f: X \subseteq K \to K$ be a function definable in a strictly $P$-minimal structure $(K, \Lm)$. There exist definable disjoint subsets $U, V$ of $X$, with $X \setminus (U \cup V)$ finite, such that 
\begin{itemize} 
	\item[(a)] $f$ is continuous on $U \cup V$,
	\item[(b)] there exists a finite partition of $U$ into sets $U_i$, such that $\restr{f}{U_i}$ is constant,
	\item[(c)] $f$ is locally strictly monotone on $V$.
\end{itemize}
\end{theorem}

In section \ref{sec:mainproofs}, we show that our main results hold for $p$-adic fields (i.e. finite field extensions of $\Q_p$). As a next step, we
generalize to definable families of functions in Section \ref{sec:def_fam}. This will allow us to deduce the validity of our results for the wider class of strictly $P$-minimal structures.

\subsection*{Acknowledgements} The authors would like to thank Raf Cluckers for his advice and encouragement during the preparation of this paper. 

\section{Proofs of the main results}

We start with some observations on $P$-minimal functions. First, it is easy to see that the following lemma, which was originally proven by Denef \cite[Lemma 7.1]{denef-84} for semi-algebraic sets, is in fact valid for $P$-minimal structures in general.
\begin{lemma}[Denef] \label{denef}
Let $S \subseteq K^{m+q}$ be a set definable in a $P$-minimal structure $(K, \Lm)$. Let $\pi_m: K^{m+q} \to K^m$ denote the projection onto the first $m$ coordinates. \item Assume there exists $M \geq 1$ such that for all $y \in \pi_m(S)$, the fibers $\pi_{m}^{-1}(y)$ are nonempty and contain at most $M$ points.  Then there exists a definable function  $g: \pi_m(S) \to S$, such that $(\pi_m\circ g)(y)= y$ for all $y \in \pi_m(S)$. 
\end{lemma}

One of the questions posed by Haskell and Macpherson in \cite{has-mac-97} was whether the assumption of definable Skolem functions could be eliminated from Remark 5.5 of their paper. Since they only needed Skolem functions for finite fibers of the same size, the result from Lemma \ref{denef} suffices. Therefore we have that:
\begin{lemma}\label{lemma5.5}
  Let  $f: X \subseteq K^n \to K$ be a function definable in a $P$-minimal structure $(K, \Lm)$. Let $Y$  be the set 
  \[ Y = \{y \in X  \mid f \text{  is defined and continuous in a neighbourhood of } y\},\]
 then $\dim\left(X \setminus Y\right)<\dim(X)$.
\end{lemma}
Recall that the dimension of a definable set $X\subset K^n$ is the greatest integer $k$ for which there exists a projection map $\pi:K^n\to K^k$, such that $\pi(X)$ has non-empty interior in $K^k$ (we refer to \cite{has-mac-97} for more details).

We will also need the fact that the finite fibers of a definable function $f:K \to K$ are uniformly bounded:
\begin{lemma}\label{lemma:uniformbound}
Let $f: K \to K$ be a $P$-minimal function.
There exists an integer $M_f$, such that if the fiber $f^{-1}(y)$ is finite for some $y \in f(K)$, then it contains at most $M_f$ elements.
\end{lemma}
\begin{proof} 
This follows immediately from Lemma 5.3 of \cite{has-mac-97}.
\end{proof}

\subsection{Preliminary lemmas and definitions}
Let us first show that if $K$ is a $p$-adic field, then a $P$-minimal definable function $f: K \to K$ with zero derivative must be piecewise constant.
\begin{proof}[Proof of Lemma \ref{lemma_locally_constant}]
By $P$-minimality, the domain of $f$  is a finite union of points and open sets, so we may as well assume that $\mathrm{dom}(f)$ is an  open set $U$. Fix $\epsilon>0$. For every $x_0\in U$, the fact that $f'(x_0)=0$ implies that there exists $\delta_{x_0}>0$, such that for all $t$ with $\abs{t}<\delta_{x_0}$,
	\begin{equation}\label{vgl_deriv_zero}
		\abs{ f(x_0+t)-f(x_0)} < \epsilon\abs{t}<\epsilon\delta_{x_0}.
	\end{equation}
	Note that we may assume that $\delta_{x_0}\in \abs{K}$. Since every open set in $K$ can be covered by a countable number of disjoint balls, we can write $U = \bigcup_{i=1}^{\infty}B(x_i,\delta_{x_i})$. Formula \eqref{vgl_deriv_zero} implies that $f(B(x_i,\delta_{x_i}))$ is contained in a ball with radius $\epsilon\delta_{x_i}$. Let $\mu$ be the Haar measure on $K$, normalized such that $\mu(\cO_K)=1$. Clearly, $\mu(B(x,\delta))=\delta$ if $\delta\in \abs{K}$. Now estimate the volume of $f(U)$:
\[\mu(f(U)) \leq \sum_{i=1}^{\infty}\mu(f(B(x_i,\delta_{x_i}))) \leq \sum_{i=1}^{\infty}\epsilon\delta_{x_i} = \epsilon \mu\biggl(\bigcup_{i=1}^{\infty}B(x_i,\delta_{x_i})\biggr),\]
hence $\mu(f(U))\leq\epsilon\mu(U)$. Since the choice of $\epsilon>0$ was arbitrary, we conclude that $f(U)$ has measure zero and hence, by $P$-minimality,  is a finite set. One can then partition the domain into a finite union of points and open sets, on each of which the image is constant.
\end{proof}
\begin{lemma}\label{lemma_inj_cons}
	Let $K$ be a $p$-adic field and let $f:X\subset K\to K$ be a function definable in a $P$-minimal structure. There exists a finite partition of $X$ in definable sets $X = \cup_i X _i$ such that for each $i$, $\restr{f}{X_i}$ is either injective or constant. 
\end{lemma}
\begin{proof}
First note that the piece of the domain on which $f$ is locally constant is a definable set $X_0$, consisting of the points $x\in X$ that satisfy the formula $\phi(x)$:
\[\phi(x) \leftrightarrow (\exists y)(\exists r)(\forall z)[f(x) = y \wedge \abs{z-x}<r \rightarrow f(z) = y].\]
Since $f$ is locally constant on $X_0$,  $f'(x) = 0$ on $X_0$. Applying Lemma \ref{lemma_locally_constant}, we can then partition $X_0$ into a finite number of sets, on each of which $f$ is constant.

 Now consider the set $A = X \setminus X_0$. By $P$-minimality, any fiber $f^{-1}(y)$ with $y \in f(A)$ will be finite. Moreover,  there exists an upper bound $M_f$ for the size of these fibers, because of Lemma \ref{lemma:uniformbound}. 

We can use the following procedure to partition $A$ into a finite number of sets $X_i$, such that $\restr{f}{X_i}$ is injective.

Applying Lemma \ref{denef} to the graph of $\restr{f}{A}$, we can find a definable function $g_1$ that chooses a point $x$ in every fiber $f^{-1}(y)$, for $y \in f(A)$. 
We can then put $X_1 = \{g_1(y) \mid y \in f(A) \}$. Then $\restr{f}{X_1}$ is injective by construction. 

Repeating the procedure for $A \setminus X_1$, we can construct a set $X_2$ on which $f$ is injective, and so on. Lemma \ref{lemma:uniformbound} ensures this algorithm will stop after at most $M_f$ steps, so that we indeed obtain a finite partition. 
\end{proof}
\begin{lemma}\label{derivative_zero}
	Let $K$ be a $p$-adic field and $f:X\subset K\to K$ a $P$-minimal function. There exists a finite subset $I\subset X$ such that for every $x_0$ in $X\setminus I$, with $f(x_0) = y_0$, the following holds:
	
	If $\lvert f'(x_0)\rvert = +\infty$, then $f$ is locally injective around $x_0$ and $(f^{-1})'(y_0)=0$. 
\end{lemma}
\begin{proof}
	By Lemma \ref{lemma_inj_cons}, one can partition $X$ into a finite number of sets $Y_i$, such that $\restr{f}{Y_i}$ is either injective or constant. Note that one only needs to consider those sets $Y_i$ on which $f$ is injective, since $f'=0$ if $f$ is constant. By Lemma \ref{lemma_5.1} and $P$-minimality there exists a finite set $I$ such that, if we put $X_i = Y_i\setminus I$, then $X_i$ is open, and  $f^{-1}$ is continuous on $f(X_i)$. Let $x_0\in X_i$ be such that $\abs{f'(x_0)}=+\infty$. That $\lvert f'(x_0)\rvert =+\infty$ means that for every $M>1$, there exists $\delta>0$ such that for all $t$ with $\abs{t}<\delta$,
	\begin{equation}\label{vgl_afgeleide_oneindig}
		\left\lvert \frac{f(x_0+t)-f(x_0)}{t}\right\rvert>M.
	\end{equation}
	Now if $(f^{-1})'(y_0)\neq 0$, then there exist $\epsilon>0$ and $s$ arbitrarily close to $0$ such that
	\begin{equation}\label{vgl_afgeleide_niet_nul}
		\left\lvert\frac{f^{-1}(y_0+s)-f^{-1}(y_0)}{s}\right\rvert>\epsilon.
	\end{equation}
	Now choose $M = 1/\epsilon$ and let $\delta$ be such that (\ref{vgl_afgeleide_oneindig}) holds for $M$. By the continuity of $f^{-1}$ around $y_0$, for $s$ close enough to $0$, $f^{-1}(y_0+s)$ lies in $B(x_0,\delta)$. Therefore $f^{-1}(y_0+s)=x_0+t$ for some $t$ with $\abs{t}<\delta$, and hence
	\begin{equation}\label{vgl>M}
		\left\lvert\frac{s}{t}\right\rvert =\left\lvert\frac{y_0+s-y_0}{t}\right\rvert= \left\lvert\frac{f(x_0+t)-f(x_0)}{t}\right\rvert>M,
	\end{equation}
	but then (\ref{vgl_afgeleide_niet_nul}) and (\ref{vgl>M}) imply that
	\begin{equation*}
		\epsilon<\left\lvert\frac{t}{s}\right\rvert<1/M=\epsilon,
	\end{equation*}
	which is a contradiction. 
\end{proof}

To show that an $o$-minimal function $f$ is differentiable, it suffices to check that the left and right derivative of $f$ are equal. Unfortunately, in the $p$-adic case we will have to deal with more possible directions. The next proposition shows that there are only finitely many possibilities. 
\begin{proposition}\label{lemma_dim_graph}
	Let $X\subset K^n$ be a set definable in a $P$-minimal structure. Write $\overline{X}$ for the topological closure of $X$. Then $\dim(\overline{X}\setminus X)<\dim(X)$.
\end{proposition}
\begin{proof}
  	Let $F:\overline{X}\to K$ be the function taking the value $1$ on $X$ and $0$ outside $X$. Applying Lemma \ref{lemma5.5}, we find  that $\dim(\overline{X} \setminus \mathrm{int}(X) )< \dim(\overline{X})$. Since $\dim(X) = \dim( \overline{X})$, a straightforward computation now yields the required result.
\end{proof}
\begin{corollary}\label{cor_finite_values}
	Let $f:K\to K$ be a $P$-minimal function. Then for each $x_0\in K$, the limit $\lim_{t\to 0}(f(x_0+t)-f(x_0))/t$ takes only a finite number of values.
\end{corollary}
\begin{proof}
	Fix $x_0\in K$ and consider the function 
	\[g:K^\times\to K: t\mapsto \frac{f(x_0+t)-f(x_0)}{t}.\] 
	Since the graph of $g$ has dimension $1$, Proposition \ref{lemma_dim_graph} implies that the set $\overline{\Gamma(g)}\setminus\Gamma(g)$ has dimension zero. This proves the corollary, since all the limit values of $\lim_{t\to 0} (f(x_0+t)-f(x_0))/t$ lie in the projection of $\overline{\Gamma(g)}\setminus\Gamma(g)$ onto the second coordinate, which is a definable subset of $K$ with dimension zero and hence, by $P$-minimality, is finite.
\end{proof}
\begin{definition}
Fix a positive integer $n$ and an element $\lambda\in K$. Define the \emph{directional derivative} along $\lambda$ with respect to $n$ in the point $x_0\in K$ to be 
\begin{equation}\label{eqn_direct_deriv} 
	{f'}^\lambda _{(n)}(x_0) = \lim_{t\to 0,\, t\in\lambda P_n}\frac{f(x_0+t)-f(x_0)}{t},
\end{equation}
if this limit exists. If $n$ is clear from the context we will omit the index $n$ and just write ${f'}^\lambda(x_0)$. \end{definition}
The next lemma and its corollary explain why it suffices to consider these directional derivatives.
\begin{lemma} Let $(K,\Lm)$ be a $P$-minimal structure, $K$ a $p$-adic field, and let $f: K \to K$ be a $P$-minimal function. For every $x_0 \in K$, there exists $n_0 \in \N$, such that for all $n \geqslant n_0$ and all $\lambda \in K$, either the limit ${f'}^\lambda _{(n)}(x_0)$ exists, or $\abs{{f'}^\lambda _{(n)}(x_0)} = +\infty$.
\item Moreover, given any sequence $(t_j)$ with $\lim t_j\to 0$, for which the limit $L= \lim_{j\to \infty}  \frac{f(x_0 + t_j) - f(x_0)}{t}$ exists, there exists $n \in \N$ and $\lambda \in K$ such that this limit $L$ equals $f^{\prime\lambda}_{(n)}(x_0)$.
\end{lemma}
\begin{proof}
Fix $x_0$, and let $g$ be the quotient function $g(t) = \frac{f(x_0+t)-f(x_0)}{t}$.
By Corollary \ref{cor_finite_values}, there exist only finitely many values $y_i$ for which there is a sequence $(t_{j}^{(i)})$ such that $g(t_{j}^{(i)}) \to y_i$ if $t_{j}^{(i)} \to 0$. Choose disjoint balls $B_i$, each containing exactly one of the limit points $y_i$. Let $B$ be a ball with center $0$. Now put $D_i = g(B)\cap B_i$, and $D = g(B) \setminus \cup_i B_i$, so that the sets $g^{-1}(D)$ and $g^{-1}(D_i)$ form a finite partition of $B$ into definable sets.

Clearly, if the sequence $g(t_j^{(i)})$ tends to $y_i$, then (the tail of) the sequence $(t_j^{(i)})$ is contained in $g^{-1}(D_i)$. Similarly, the only sequences contained in $g^{-1}(D)$ are those for which $\abs{g(t_j)} \to +\infty$. To see this, consider a sequence $(t_j)$ with $t_j \to 0$, contained in $g^{-1}(D)$, and assume that $\abs{g(t_j)}$ is bounded for all $j$. 
 Then the set $G= \{g(t_j) \mid j\in \N\}$ is a bounded set, which can be assumed to be infinite. Our assumptions on $K$ imply that the valuation ring $R \subset K$ is compact, and hence the closure $\overline{G}$ must be compact, since for some $m \in \Z$, it is a closed subset of the compact set $\pi^mR$. Therefore, $\overline{G}$ must contain a limit point, which must necessarily be one of the points $y_i$. Since $G \cap (\cup_iB_i) = \emptyset$ by construction, we obtain a contradiction.

Each set $g^{-1}(D)$ or $g^{-1}(D_i)$ can be partitioned in cells $C$. In this way we also get a cell decomposition of $B$. It is easy to check that if $0 \in \overline{C}$, and if we choose $\gamma \in \Gamma_K$ big enough, then for some $n_0 \in \N$ and $\lambda \in \Lambda_{n_0}$, \[C \cap \{\ord(x)> \gamma\} = \lambda P_{n_0} \cap \{\ord(x) > \gamma\},\]  implying that $C$ and $\lambda P_{n_0}$ contain the same sequences converging to 0. (Note that we can use the same value of $n_0$ in all cells). Since the sets $g(C)$, by construction, contain at most one of the points $y_i$, the limits ${f'}_{(n_0)}^{\lambda}(x_0)$ must either be well defined, or 
$\abs{{f'}_{(n_0)}^{\lambda}(x_0)} = +\infty$, and the same obviously holds for all $n\geqslant n_0$.
\end{proof}
\begin{corollary}\label{cor_final}
Let $(K,\Lm)$ be a $P$-minimal structure, $K$ a $p$-adic field, and let $f: K \to K$ be a $P$-minimal function. If for some $x_0 \in K$, the derivative $f'(x_0)$ does not exist, then either there are $\lambda,n$ such that $\abs{f^{\prime\lambda}_{(n)}(x_0)} = +\infty$, 
or, if all directional derivatives are bounded, there exist $n, \lambda, \mu$ such that 
$f^{\prime\lambda}_{(n)}(x_0) \neq f^{\prime\mu}_{(n)}(x_0)$.
\end{corollary}

\subsection{Proofs of the main results (for $p$-adic fields)} \label{sec:mainproofs}
Throughout this section we will assume that we work in a $P$-minimal structure $(K,\Lm)$ and that $K$ is a $p$-adic field. Also, $f$ will always denote a $P$-minimal function. The main step in the proof of Theorem \ref{thm_bijna_overal_afleidbaar} will be to show that sets of the following type are finite.
\begin{definition}\label{def_Sn_Tn}
	For every positive integer $n$ we define
	\begin{align*}S_n &= \left\{x_0\in K\left| \begin{array}{l}\text{the limit $ {f'}_{(n)}^{\lambda}(x_0)$ exists for all $\lambda$ in $\Lambda_n$,} \\ \text{and there exist $\lambda,\mu\in\Lambda_n$ such that }{f'}_{(n)}^{\lambda}(x_0)\neq {f'}_{(n)}^{\mu}(x_0)\end{array}\right\}\right.
		\intertext{and}
		T_n &= \bigl\{x_0\in K\mid \text{ there exists $\lambda\in\Lambda_n$ such that }\abs{{f'}_{(n)}^{\lambda}(x_0)}=+\infty\bigr\}.
\end{align*}
\end{definition}
In order to prove Theorem \ref{thm_bijna_overal_afleidbaar}, it will be sufficient to show that both $\cup_{n}S_n$ and $\cup_{n}T_n$ are finite, because of Corollary \ref{cor_final}. 
\begin{lemma}\label{lemma_Sn_finite}
The set $S_n$ is finite for every $n>0$. 
\end{lemma}
\begin{proof}
	Assume that $S_n$ is infinite for some $n>0$. By $P$-minimality it must then contain a ball $B$. By Lemma \ref{lemma_5.1}, after shrinking $B$ if necessary, we may assume that for every $\lambda\in\Lambda_n$, ${f'}^{\lambda}$ is continuous on $B$. 
	
	Fix $x_0\in B$. By the definition of $S_n$, there exist $\lambda,\mu\in\Lambda_n$ such that ${f'}^{\lambda}(x_0)\neq {f'}^{\mu}(x_0)$. After replacing $f$ by $f(x)-{f'}^{\lambda}(x_0)\cdot x$ and rescaling, we may assume that ${f'}^{\lambda}(x_0) = 0$ and ${f'}^{\mu}(x_0) = 1$. By Hensel's lemma, there exists $m$ such that $1+\pi^m\mathcal{O}_K\subset P_n$. Fix $0<\epsilon<\lvert \pi^m\rvert$. Because ${f'}^{\mu}$ is continuous, the following conditions hold if we choose $t_{\lambda}\in\lambda P_n$ and $t_{\mu}\in\mu P_n$ to be small enough:	\begin{equation}\label{vgl_t_lambda}
		\lvert f(x_0+t_{\lambda})-f(x_0)\rvert<\epsilon\lvert t_{\lambda}\rvert,
	\end{equation}
\begin{equation}\label{vgl_t_mu}
	\lvert f(x_0+t_{\mu})-f(x_0)\rvert=\lvert {f'}^{\mu}(x_0)\rvert\lvert t_{\mu}\rvert,
	\end{equation}
\begin{equation}\label{vgl_cont}
	\lvert {f'}^{\mu}(x_0+t_{\lambda})\rvert = \lvert {f'}^{\mu}(x_0)\rvert = 1.
	\end{equation}
	By changing our choices for $t_\lambda$ and $t_\mu$ (choosing a smaller $\epsilon$ if necessary) we can moreover assume that $\lvert t_{\mu}\rvert = \epsilon\lvert t_{\lambda}\rvert$. By our choice of $m$ and $\epsilon$, we have that $t_{\lambda}+t_{\mu} = t_\lambda(1+\frac{t_\mu}{t_\lambda})\in \lambda P_n(1+\pi^m\mathcal{O}_K)\subset \lambda P_n$. By (\ref{vgl_cont}), equation (\ref{vgl_t_mu}) also holds for $x_0$ replaced by $x_0+t_{\lambda}$, so that
\begin{equation}\label{vgl_contradictie}
	\lvert f(x_0+\tl+\tm)-f(x_0+\tl)\rvert = \lvert {f'}^{\mu}(x_0+\tl)\rvert\lvert\tm\rvert = \lvert \tm\rvert.
\end{equation}
On the other hand, since $t_\lambda+t_\mu\in \lambda P_n$, $\tl$ can be replaced by $\tl+\tm$ in \eqref{vgl_t_lambda}, so
\begin{equation*}
	\lvert f(x_0+\tl+\tm)-f(x_0)\rvert<\epsilon\lvert\tl+\tm\rvert = \epsilon\lvert \tl\rvert.
\end{equation*}
But then $\lvert f(x_0+\tl+\tm)-f(x_0+\tl)\rvert$ is equal to
\begin{equation*} \lvert\left(f(x_0+\tl+\tm)-f(x_0)\right)-\left(f(x_0+\tl)-f(x_0)\right)\rvert<\epsilon\lvert\tl\rvert=\lvert\tm\rvert.
\end{equation*}
This contradicts (\ref{vgl_contradictie}), which finishes the proof.
\end{proof}
\begin{corollary}\label{cor_USn_finite}
	The set $\cup_n S_n$ is finite.
\end{corollary}
\begin{proof}
	By Lemma \ref{lemma_Sn_finite}, $\cup_n S_n$ is countable. It therefore suffices to show that $\cup_n S_n$ is definable, because in a $P$-minimal structure every countable, definable subset of $K$ is finite. The following formula $\psi(x)$ expresses that all the directional derivatives are bounded:
\begin{align*}
\psi(x) &\leftrightarrow (\exists t_1,t_2)(\forall z)\biggl[0<\abs{ z} <\abs{ t_1} \rightarrow \left\lvert\frac{f(x+z)-f(x)}{z}\right\rvert <\abs{ t_2}\biggr].\\
\intertext{The formula $\phi(x)$ expresses that ${f'}(x)$ does not exist:}
\phi(x) &\leftrightarrow\neg(\exists L)(\forall t_1)(\exists t_2)(\forall z)\biggl[0<\abs{ z} <\abs{ t_2} \rightarrow \left\lvert\frac{f(x+z)-f(x)}{z}-L\right\rvert <\abs{ t_1}\biggr].
\end{align*}
Hence $\cup_n S_n$ is defined by the formula $\psi(x)\wedge \phi(x)$, because of Corollary \ref{cor_final}.
\end{proof}
\begin{lemma}\label{lemma_Tn_finite}
	The set $T_n$ is finite for every $n>0$.
\end{lemma}
\begin{proof}
	We write $T_n = T_n^0 \,\cup\, T_n^\infty$ where
\begin{align*}
	T_n^0 &=\{x_0\in K\mid \exists \lambda,\mu\in\Lambda_n: \abs{ {f'}_{(n)}^\lambda (x_0)}=+\infty\text{ and } \abs{ {f'}_{(n)}^\mu (x_0)}<+\infty\}\\
	\intertext{and}
	T_n^\infty &= \{x_0\in K\mid \forall\lambda\in\Lambda_n: \abs{ {f'}_{(n)}^\lambda (x_0)}=+\infty\}.
\end{align*}
Fix $n>0$. To simplify notation, we will omit the index $n$ and just write ${f'}^\lambda$. The proof of the finiteness of $T_n^0$ is very similar to the proof of the corresponding result for $S_n$. Therefore we only indicate the differences. After rescaling $f$ we may assume that $\abs{ {f'}^\lambda (x_0)}=+\infty$ and ${f'}^\mu (x_0)=1$. Formula \eqref{vgl_t_lambda} should be replaced by $\abs{f(x_0+t_\lambda)-f(x_0)}>M\abs{t_\lambda}$, for a fixed $M>\abs{\pi^{-m}}$, where $m$ is as before. The remainder of the proof is left as an exercise.

 Now suppose $T_n^{\infty}$ were infinite. By $P$-minimality, this set must contain a ball $B$ on which $\abs{f'}=+\infty$. By Lemma \ref{derivative_zero}, we may assume that  $f$ is injective on $B$ and that $(\restr{f^{-1}}{f(B)})'=0$, after shrinking $B$ if necessary. Lemma \ref{lemma_locally_constant} then implies that $\restr{f^{-1}}{f(B)}$ is locally constant, which is clearly impossible.
\end{proof}
\begin{corollary}\label{cor_UTn_finite}
	The set $\cup_n T_n$ is finite.
\end{corollary}
\begin{proof}
	As in the proof of Corollary \ref{cor_USn_finite}, we need to verify that $\cup_nT_n$ is definable. But this is clearly the case, since one can use the formula $\neg\psi(x)$, where $\psi(x)$ is as in the proof of Lemma \ref{cor_USn_finite}.
\end{proof}
\begin{proof}[Proof of Theorem \ref{thm_bijna_overal_afleidbaar} (for $p$-adic fields)]
	By Corollary \ref{cor_USn_finite}, Corollary \ref{cor_UTn_finite} and the discussion right after Definition \ref{def_Sn_Tn}, we know that there is a cofinite, definable set $A\subset K$ on which $f$ is differentiable. Since $A$ is definable, it is a finite union of points and open sets, namely $A = \bigcup_{i=1}^n A_i \cup \bigcup_{i=1}^k\{a_i\}$. 
	
	Applying Theorem \ref{thm_71.2} yields that the definable set \[A_i' = \{x\in A_i\mid f\text{ is strictly differentiable at }x\}\] is dense in $A_i$, for $i=1,\ldots,n$. But then the sets $I_i = A_i\setminus A_i'$ cannot contain any balls, and hence they are finite by $P$-minimality. So $A\setminus\bigcup_{i=1}^n I_i$ is a cofinite set on which $f$ is strictly differentiable.
\end{proof}
We are now ready to give a proof of Theorem \ref{thm_LJP} for $p$-adic fields. For the sake of clarity we will restate the theorem.
\begin{theorem*}[Local Jacobian Property]
	Let $K$ be a $p$-adic field and $f:X\subset K \to K$ a $P$-minimal function. There exists a finite set $I\subseteq X$, and a finite partition of $X\setminus I$ into definable open sets $X_i$, such that either $\restr{f}{X_i}$ is constant on $X_i$, or the following holds on $X_i$: for every $x$ in $X_i$, there is an open ball $B\subset X_i$ containing $x$, such that the map $\restr{f}{B}$ satisfies the following properties:
	\begin{enumerate}
		\item[(a)] \label{a}$\restr{f}{B}$ is a bijection, and $f(B)$ is a ball,
		\item[(b)] \label{b}$f$ is strictly differentiable on $B$ with strict derivative $Df$,
		\item[(c)]\label{c}$\lvert Df\rvert$ is constant on $B$,
		\item[(d)] \label{d} for all $x,y \in B$, one has that $\lvert Df\rvert\lvert x-y\rvert = \lvert f(x)-f(y)\rvert$.
	\end{enumerate}
\end{theorem*}
\begin{proof}
	By Theorem \ref{thm_bijna_overal_afleidbaar}, there exists a finite set $I\subset X$ such that $f$ is strictly differentiable on $X \setminus I$. This proves (b).
	Put $X\setminus I = A_0 \, \cup \, A$, with $A_0 = \{x \in X \setminus I \mid Df(x) = 0\}$, and $A = \{x \in X \setminus I \mid Df(x) \neq 0\}$. 
By Lemma \ref{lemma_locally_constant}, $f$ is then piecewise constant on $A_0$. By Lemma \ref{lemma_inj_cons} we can partition $A$ in a finite number of pieces $X_i$, on which $f$ is injective. Moreover, by $P$-minimality, one can assume that each $X_i$ is open (after excluding a finite number of points if necessary).

It remains to check that $(a), (c), (d)$ hold for all points of $X_i$.
Part $(c)$ and $(d)$ are immediate consequences of Theorem \ref{thm_bijna_overal_afleidbaar}: pick any $a \in X_i$. There exists a ball $B \subset X_i$ such that for all $x,y \in B$, it holds that
\[\left\lvert\frac{f(x)-f(y)}{x-y}\right\rvert = \abs{Df(a)} .\] Consequently, we must have that $\abs{Df(a)} = \abs{Df(a')}$ for all $a' \in B$ (since $B$ contains a neighborhood of $a'$), from which $(c)$ and $(d)$ follow.\newline
That $(a)$ holds can be seen as follows (this part of the proof is inspired by Lemma 27.4 of \cite{sch-84}). Fix any $a \in X_i$, and take a ball $B(a,r)$ which is small enough to assure that for all $x,y \in B(a,r)$,
\[ \sup \left\{\left\lvert\frac{f(x)-f(y)}{x-y} - Df(a)\right\rvert: x,y \in B, x\neq y\right\} <\abs{Df(a)}.\]
Clearly this implies that $f(B(a,r)) \subseteq B(f(a), \abs{Df(a)}r)$. It suffices to check that $\restr{f}{B(a,r)}$ is surjective.
Choose $c \in B(f(a), \abs{Df(a)}r)$. We will show that the map $x \mapsto f(x)-c$ has  a zero in $B(a,r)$. For $x \in B(a,r)$, put $g(x)= x - (f(x)-c)/Df(a)$. Then $g$ maps $B(a,r)$ into $B(a,r)$. Moreover, for all $x,y \in B(a,r)$, we have that
\begin{eqnarray*}\abs{g(x) -g(y)} &=& \left\lvert x-y-\frac{f(x) -f(y)}{Df(a)}\right\rvert\\
&=& \left\lvert\frac{x -y}{Df(a)}\right\rvert \left\lvert\frac{f(x) -f(y)}{x-y}-Df(a)\right\rvert\\
&\leq& \tau \abs{x-y},
\end{eqnarray*}  
for some $0<\tau<1$. Since $g: B(a,r) \to B(a,r)$ is a contraction,  the Banach fixed-point theorem yields that $B(a,r)$ contains a point $z$ for which $g(z) =z$, and hence 
$f(z)=c$.
\end{proof}
It is then easy to deduce that Local Monotonicity (as stated in Theorem  \ref{thm:mono}) holds for $p$-adic fields.

We can use the techniques from the proof of the Local Jacobian Property to obtain a generalisation of Lemma \ref{lemma_locally_constant}. This is probably not new, but since we could not find any reference, we give a proof in full detail. Let $\mu$ be the Haar measure on $K$, normalized such that $\mu(\mathcal{O}_K)=1$. We use the notation $\mu(A)=\int_A\abs{\mathrm{d}x}$ for a measurable set $A\subset K$.
\begin{proposition}\label{prop:integral}
	Let $(K,\Lm)$ be a $P$-minimal structure, $K$ a $p$-adic field. Let $X,Y\subset K$ be definable, measurable sets, and let $f:X\to Y$ be a definable bijection that is strictly differentiable. Then
	\[\mu(Y) = \int_X \abs{Df(x)}\abs{\mathrm{d}x},
	\]
	where the equation holds in $\mathbb{R}\cup\{+\infty\}$.
\end{proposition}
\begin{proof}
	Partition $X=X_0\cup X_1$, where $X_0 = \{x\in X\mid Df(x)=0\}$ and $X_1 = \{x\in X\mid Df(x)\neq 0\}$. We proved in Lemma \ref{lemma_locally_constant} that $\mu(f(X_0))=0$, so we may just as well assume that $Df$ is nonzero on all of $X$. Also, we can assume that $X$ is open, after excluding a finite number of points if necessary. Since $K$ is a $p$-adic field, $X$ can be partitioned into a countable union of disjoint balls $B_i$, such that $\abs{Df} = \abs{c_i}$ is constant on $B_i$ and such that $\abs{f(x)-f(y)}=\abs{c_i}\abs{x-y}$ for all $x,y\in B_i$. As in the proof of the Local Jacobian Property, we can argue that $f$ maps $B_i$ bijectively onto a ball $B'_i$ and it can be seen easily that $\mu(f(B_i)) = \abs{c_i}\mu(B_i)$. Since the integrand takes non-negative values, we can use sigma-additivity to compute
\[\int_X \abs{Df(x)}\abs{\mathrm{d}x} = \sum_i\int_{B_i}\abs{Df(x)}\abs{\mathrm{d}x}=\sum_i\abs{c_i}\mu(B_i)=\mu(Y),
\]
which proves the formula.
\end{proof}
\subsection{Generalisation to strictly $P$-minimal structures}\label{sec:def_fam}
Given a definable function $f: A \times K \subseteq K^{n+1} \to K$, we write $\{f_{\alpha}\}_{\alpha \in A}$ for the family of functions whose members are defined by putting $f_\alpha(x) = f(\alpha, x)$. Our results can be generalized to this setting.

For a set $S \subseteq K^{n+1}$, we let $S_{\alpha}$ denote the fiber $S_\alpha = \{x\in K\mid (\alpha,x)\in S\}$. \begin{theorem}[Strict differentiation for definable families]\label{thm_bijna_overal_afleidbaar_param}
	\hspace{-1mm}Let $K$\hspace{-0.4mm} be a $p$-adic field and $f:A\times K\subset K^{n+1}\to K$ a $P$-minimal function. There exists a definable set $S\subset A\times K$ such that for each $\alpha\in A$, $S_\alpha$ is a cofinite subset of $K$ and $f_\alpha$ is strictly differentiable on $S_{\alpha}$.
\end{theorem}
\begin{proof}
The strict derivative $Df_{\alpha}(a)$ of $f_{\alpha}$ in a point $a$ can be considered as the partial (strict) derivative
\[ \lim_{(x,y) \to (a,a)} \frac{f(\alpha, x) - f(\alpha, y)}{x-y}\]
of $f$ with respect to the last variable. Let $S$ be the set consisting of points $(\alpha,a)\in A\times K$ such that $f_\alpha$ is strictly differentiable in $a$. It is easy to see that this  is a definable set (the definition is similar to the formula $\phi$ given in the proof of Corollary \ref{cor_USn_finite}). The fact that for each $\alpha\in A$, $S_\alpha$ is a cofinite set, is a direct application of Theorem \ref{thm_bijna_overal_afleidbaar}. 
\end{proof}

Next, we present a uniform version of the Local Monotonicity Theorem for $p$-adic fields. Given sets $S = A \times K$ and $D \subset A$, we write $S^{(D)}$ for the set $S^{(D)} = \{( \alpha, x) \in S \mid \alpha \in D\}$. The fibers will be denoted as $S^{(D)}_{\alpha}$, for $\alpha \in D$.
\begin{theorem}[$p$-adic Local Monotonicity in definable families]\label{thm:mono-family}
	Let $K$ be a $p$-adic field and $f:A\times K\subset K^{n+1}\to K$ a $P$-minimal function. Then there exist definable disjoint subsets $U,V$ of $A\times K$ such that for each $\alpha\in A$ the following conditions hold:
\begin{enumerate}
	 \item $K\setminus (U_\alpha\cup V_\alpha)$ is finite,
	 \item $f_\alpha$ is continuous on $U_\alpha\cup V_\alpha$,
	 \item $f_{\alpha}$ is piecewise constant on $U_{\alpha}$. More specifically, there exists a finite partition of $U$ in definable sets $U_i$, such that for each $\alpha\in A$, the function $f_\alpha$ is constant on each of the fibers $(U_i)_\alpha$,
	\item $f_{\alpha}$ is locally strictly monotone on $V_{\alpha}$.
\end{enumerate}
\end{theorem}
\begin{proof}
	By Theorem \ref{thm_bijna_overal_afleidbaar_param}, there exists a definable subset $S\subset A\times X$ such that for each $\alpha\in A$, $S_\alpha$ is a cofinite set on which $f_\alpha$ is strictly differentiable. Let $U$ be the set of all points $(\alpha,x)\in S$ such that $Df_\alpha(x)=0$, and put $V=S\setminus U$.  This proves $(1)$ and $(2)$. Part $(4)$ holds as a direct consequence of the Local Jacobian Property applied to $\restr{f_\alpha}{V_\alpha}$.
	
	By Lemma \ref{lemma_locally_constant}, there exists a finite partition of $U_{\alpha}$, such that $f_{\alpha}$ is constant on each part. We will now show that this partition can be taken uniformly in the parameter $\alpha$.

	Let $S_{\text{Im}, \alpha} = \{ y\in K \mid \exists x: y = f_{\alpha}(x)\}$. For each $\alpha\in\pi_n(U)$, where $\pi_n$ denotes the projection on the first $n$ coordinates, this set is finite by Lemma \ref{lemma_locally_constant}. Applying Lemma 5.3 of \cite{has-mac-97} yields that there exists a partition of $\pi_n(U)$ into sets $A_1, \ldots A_k$, and an integer $M$ such that for $\alpha \in A_i$, the set $S_{\text{Im}, \alpha}$ contains at most $M$ elements. Now Lemma \ref{denef} asserts that there is a definable way to choose an element $y_0(\alpha)$ from each $S_{\text{Im}, \alpha}$. So the fibers $f^{-1}(y_0(\alpha))$ (on which $f_{\alpha}$ is constant) are uniformly definable. Repeat the process for the sets $S_{\text{Im}, \alpha}\setminus \{y_0(\alpha)\}$, and so on. The algorithm stops after at most $M$ steps. This concludes the proof of $(3)$.
\end{proof}

The above generalizations to families of definable functions imply that the Local Monotonicity Theorem is valid for any strictly $P$-minimal structure. 
\begin{proof}[Proof of Theorem \ref{thm:mono}]

Let $(K, \Lm)$ be a strictly $P$-minimal structure, and $f: K \to K$ an $\Lm$-definable function.

If the definition of $f$ contains field parameters, one can replace these by variables $\alpha$, and consider $f$ to be a member of a family $\{{g}_{\alpha}\}_{\alpha \in K^n}$, which is defined by a parameter-free formula $\psi(\alpha, x,y)$. This formula $\psi$ can then be interpreted in any $\mathcal{L}$-structure. 

By our assumption, there exists a finite extension $K'$ of $\Q_p$ which has the same $\Lm$-theory as $K$. 
We have already shown that the Local Monotonicity Theorem is valid for families of functions over $K'$, so one only needs to check  that there exists an $\Lm$-sentence asserting this fact. (As $K$ and $K'$ have the same $\Lm$-theory, this will imply that the theorem also holds for the original family $\{g_{\alpha}\}_{\alpha \in K^n}$. Since $f$ is a member of this family, this proves that the Local Monotonicity Theorem holds for $f$.)

It is clear that parts (2), (3) and (4) of Theorem \ref{thm:mono-family} can be expressed using a first order-formula.
For part (1), one can use the fact that in a $P$-minimal structure, a definable set is cofinite if and only if its complement does not contain a ball. This clearly is a first-order condition. 
\end{proof}
By the same reasoning, the proofs of Theorem \ref{thm_bijna_overal_afleidbaar} and Theorem \ref{thm_LJP} can also be generalized to strictly $P$-minimal structures. Therefore, Theorem \ref{thm_bijna_overal_afleidbaar_param} and Theorem \ref{thm:mono-family} also hold for strictly $P$-minimal structures.

\bibliographystyle{plain}
\bibliography{bibliography}
\noindent
TRISTAN KUIJPERS\\
KU LEUVEN\\
DEPARTMENT OF MATHEMATICS\\
CELESTIJNENLAAN 200B\\
3001 LEUVEN, BELGIUM\\
{\it E-mail}: tristan.kuijpers@wis.kuleuven.be\\[8pt]
EVA LEENKNEGT\\
PURDUE UNIVERSITY\\
DEPARTMENT OF MATHEMATICS\\
150 N. UNIVERSITY STREET\\
WEST LAFAYETTE, IN 47907-2067\\
{\it E-mail}: eleenkne@math.purdue.edu\\eva.leenknegt@gmail.com

\end{document}